\documentclass[a4paper,11pt]{amsart}
\usepackage{hyperref,fancyhdr,mathrsfs,amsmath,amscd,amsthm,amsfonts,latexsym,amssymb,stmaryrd}
\usepackage[all]{xy}

\DeclareMathOperator{\dif}{d}

\newcommand{\Cal}{\mathcal{C}}

\newcommand{\ol}{\mathcal{O}}

\def \o{\omega}

\def \phi{\varphi}
\def \Phi{\varPhi} 
\def \Psi{\varPsi} 
\def \p{\pi}
\def \r{\rho}
\def \s{\sigma}
\def \t{\tau}

\def \R{\mathbb{R}}

\def \C{\mathbb{C}\,}

\def\widecheckg{g^{\hspace*{-2.5pt}\vbox to 5pt{\hbox to
0pt{\LARGE$\check{}$}}}\hspace*{2pt}}

\def\widecheckl{\lambda^{\hspace*{-3.5pt}\vbox to 8pt{\hbox to
0pt{\LARGE$\check{}$}}}\hspace*{2pt}}

\begin{document}

\title{Quaternionic-like manifolds\\ 
and homogeneous twistor spaces}
\author{Radu Pantilie} 
\thanks{The author acknowledges partial financial support from the Romanian National Authority for
Scientific Research, CNCS-UEFISCDI, project number PN-II-ID-PCE-2011-3-0362}
\email{\href{mailto:radu.pantilie@imar.ro}{radu.pantilie@imar.ro}}
\address{R.~Pantilie, Institutul de Matematic\u a ``Simion~Stoilow'' al Academiei Rom\^ane,
C.P. 1-764, 014700, Bucure\c sti, Rom\^ania}

\newtheorem{thm}{Theorem}[section]
\newtheorem{lem}[thm]{Lemma}
\newtheorem{cor}[thm]{Corollary}
\newtheorem{prop}[thm]{Proposition} 

\theoremstyle{definition}

\newtheorem{defn}[thm]{Definition}
\newtheorem{rem}[thm]{Remark}
\newtheorem{exm}[thm]{Example}

\numberwithin{equation}{section}
 
\begin{abstract} 
Motivated by the quaternionic geometry corresponding to the homogeneous complex manifolds endowed with (holomorphically) embedded spheres, 
we introduce and initiate the study of the `quaternionic-like manifolds'. These contain, as particular subclasses, the CR quaternionic and the $\r$-quaternionic 
manifolds. Moreover, the notion of `heaven space' finds its adequate level of generality in this setting: 
(essentially) any real analytic quaternionic-like manifold admits a (germ) unique heaven space, which is a $\r$-quaternionic manifold. 
We, also, give a natural construction of homogeneous complex manifolds endowed with 
embedded spheres, thus, emphasizing the abundance of the quaternionic-like manifolds. 
\end{abstract} 

\maketitle 
\thispagestyle{empty} 
\vspace{-4mm} 
\begin{center}
\emph{This paper is dedicated to the 150th anniversary of the Romanian Academy.}
\end{center}

\section*{Introduction} 

\indent 
The quaternionic geometry is, informally, that part of the geometry where the points are spheres. This means that the objects 
of quaternionic geometry are parameter spaces for families of (Riemann) spheres embedded into CR manifolds, the latter, 
being called the \emph{twistor spaces} (of the corresponding objects). When the twistor spaces are complex manifolds 
and the normal bundles of the twistor spheres are nonnegative we obtain what we call $\r$-quaternionic manifolds \cite{Pan-qgfs}\,.\\ 
\indent 
This paper grew out from our aim to understand the $\r$-quaternionic manifolds $M$ whose twistor spaces are homogeneous. 
As $M$ is not necessarily homogeneous (see Example \ref{exm:nil_orbits}\,, below), one of the main tasks is to describe 
the intrinsic geometry of the corresponding (local) orbits. This is, further, motivated by the idea that, because the corresponding twistor space 
is homogeneous, each such orbit should, somehow, determine the geometry of $M$.\\ 
\indent 
We are, thus, led to consider the \emph{quaternionic-like manifolds} (note that, this expression was previously used with a different meaning) 
which, at the linear level, are given by nonconstant holomorphic maps from the sphere to the complex Grassmannian of a vector space 
(these maps are required to be compatible with the conjugations; see Section \ref{section:q-like_linear}\,, below, for details). 
Then, similarly to the classical linear quaternionic structures, the dual of a linear quaternionic-like structure 
is quaternionic-like. Furthermore, the class of quaternionic-like manifolds contains two important subclasses: the CR quaternionic manifolds of \cite{fq}\,, 
and the $\r$-quaternionic manifolds of \cite{Pan-qgfs} (in particular, the co-CR quaternionic manifolds of \cite{fq_2}\,). 
Moreover, the notion of \emph{heaven space} (see \cite{fq}\,) finds its adequate level of generality in this setting (Theorem \ref{thm:heaven_space_for_q-like}\,). 
For example, if $M$ is a $\r$-quaternionic manifold whose twistor space is homogeneous then each orbit, of the induced local action on $M$, 
is endowed with a quaternionic-like structure whose heaven space is $M$.\\ 
\indent 
We, also, give a natural construction of homogeneous complex manifolds endowed with embedded spheres (with nonnegative normal bundles) 
which is then characterized and illustrated with several classes of examples (Section \ref{section:constr_hqo}\,). 
This emphasizes the abundance of the quaternionic-like manifolds.

\section{Quaternionic-like vector spaces} \label{section:q-like_linear} 

\indent 
This section is dedicated to the study of the following notions. 

\begin{defn} \label{defn:q-like} 
A \emph{linear quaternionic-like structure} on a (real) vector space $U$ is a nonconstant holomorphic map from the sphere 
to the complex Grassmannian ${\rm Gr}_k\bigl(U^{\C\!}\bigr)$ which intertwines the antipodal map and the conjugation determined by 
the conjugation of $U^{\C\!}$, where $k\in\mathbb{N}$\,, $k\leq\dim U$.\\ 
\indent 
A \emph{quaternionic-like vector space} is a vector space endowed with a linear quaternionic-like structure.  
\end{defn} 

\indent 
If $U$ is endowed with a linear quaternionic-like structure, we denote by $U^z$ the subspace of $U^{\C\!}$ 
which is the image of $z\bigl(\in S^2\bigr)$ through the corresponding map. Then the \emph{dual} linear quaternionic-like structure on $U^*$ is 
given by associating to each $z$ the annihilator of $U^z$. Also, a \emph{quaternionic-like linear map},  
between the quaternionic-like vector spaces $U$ and $V$, is a pair $(\psi,T)$\,, where $\psi:U\to V$ is linear 
and $T$ is an orientation preserving isometry of the sphere 
such that $\psi\bigl(U^z\bigr)\subseteq V^{T(z)}$, for any $z$\,.  

\begin{prop} \label{prop:basic_q-like_corresp} 
There exists a functorial (bijective) correspondence between linear qua\-ter\-ni\-o\-nic-like structures on $U$  
and nontrivial (that is, nonsplitting) exact sequences 
$$0\longrightarrow\mathcal{U}_-\longrightarrow t\times U^{\C\!}\longrightarrow\mathcal{U}_+\longrightarrow0$$  
of holomorphic vector bundles over the sphere $t$\,, endowed with conjugations covering the antipodal map.\\ 
\indent 
Furthermore, under this correspondence, dual linear quaternionic-like structures correspond to dual exact sequences.  
\end{prop} 
\begin{proof} 
If $U$ is endowed with a linear quaternionic-like structure then $\mathcal{U}_-$ is the pull-back through the corresponding map 
of the tautological vector bundle over the Grassmannian. 
\end{proof} 

\indent 
Note that, if $E$ is a (classical, nontrivial) quaternionic vector space (see, for example, \cite{fq}\,) and $t$ is the corresponding space  
of compatible linear complex structures then $t\to{\rm Gr}_{2k}\bigl(E^{\C\!}\bigr)$\,, $J\mapsto{\rm ker}(J+{\rm i})$\,, $(J\in t)$\,, 
is a linear quaternionic-like structure, where $\dim E=4k$\,.\\ 
\indent 
More generally, we have the following classes of examples. 

\begin{exm} 
Let $E$ be a quaternionic vector space and let $\r:E\to U$ be a linear map such that $({\rm ker}\r)\cap J({\rm ker}\r)=\{0\}$\,, 
for any compatible linear complex structure $J$ on $E$. Then, on denoting by $t$ the space of compatible linear complex structures on $E$, 
we obtain the linear quaternionic-like structure on $U$ given by $t\to{\rm Gr}_{2k}\bigl(U^{\C\!}\bigr)$\,, $J\mapsto\r\bigl({\rm ker}(J+{\rm i})\bigr)$\,, 
$(J\in t)$\,, where $\dim E=4k$\,.\\ 
\indent 
We call $(E,\r)$ a \emph{linear $\r$-quaternionic structure} \cite{Pan-qgfs} on $U$. Dualizing, we obtain what we shall call a 
\emph{linear $\r^{*\!}$-quaternionic structure}. 
\end{exm} 

\indent 
Note that, if we assume $\r$ surjective then we obtain the notions (dual to each other) of linear co-CR quaternionic and 
linear CR quaternionic structure, introduced in \cite{fq}\,.\\  
\indent 
To understand how the linear $\r$-quaternionic structures are placed among the linear qua\-ter\-ni\-o\-nic-like structures, 
we use Proposition \ref{prop:basic_q-like_corresp}\,. Then a quaternionic-like vector space $U$ is $\r$-quaternionic if and only if 
in the Birkhoff--Grothendieck decomposition of $\mathcal{U}_-$ appear only terms of Chern number $-1$\,; consequently, $U$ is, in this case, 
determined by $\mathcal{U}_+$  which we call its holomorphic vector bundle \cite{Pan-qgfs} (cf.\ \cite{fq}\,, \cite{fq_2}\,, \cite{vq}\,).\\ 
\indent 
Moreover, we have the following basic result (in which the relevant morphisms are defined in the obvious way, and $\mathcal{U}_{\pm}$ 
are as in Proposition \ref{prop:basic_q-like_corresp}\,). 

\begin{thm} \label{thm:basic_q-like_corresp} 
There exist functorial correspondences between the following, where $U$ is a vector space:\\ 
\indent 
\quad{\rm (i)} Linear quaternionic-like structures on $U$.\\ 
\indent 
\quad{\rm (ii)} Pairs $(U_+,\psi_+)$\,, where $U_+$ is endowed with a linear $\r$-quaternionic structure $(E_+,\r_+)$\,, 
and $\psi_+:U\to U_+$ is a linear map such that $\bigl(U_+\bigr)^z+\bigl({\rm im}\,\psi_+\bigr)^{\C\!}=\bigl(U_+\bigr)^{\C\!}$, 
for any $z$ in the embedded sphere of $E_+$\,.\\ 
\indent 
\quad{\rm (iii)} Pairs $(U_-,\psi_-)$\,, where $U_-$ is endowed with a linear $\r^{*\!}$-quaternionic structure $(E_-,\r_-^*)$\,, 
and $\psi_-:U_-\to U$ is a linear map such that $\bigl(U_-\bigr)^z\cap\bigl({\rm ker}\,\psi_-\bigr)^{\C}=\{0\}$\,, for any $z$ 
in the embedded sphere of $E_-$\,.\\ 
\indent 
Moreover, if $(U_+,\psi_+)$ and $(U_-,\psi_-)$ correspond to the same linear quaternionic-like structure then $E_+=E_-$ 
and $\psi_+\circ\psi_-=\r_+\circ\r_-^*$\,. Also, the holomorphic vector bundles of $(U_+,E_+,\r_+)$ and $(U_-,E_-,\r_-^*)$ 
are $\mathcal{U}_+$ and $\mathcal{U}_-$\,, respectively. 
\end{thm} 

\indent 
To prove Theorem \ref{thm:basic_q-like_corresp}\,, we shall need the following result (cf.\ \cite[Proposition 3.1]{Qui-QJM98}\,).   

\begin{prop} \label{prop:from_Quillen_3.1}
Any nonnegative holomorphic vector bundle $\mathcal{U}$ over the sphere corresponds uniquely (up to canonical isomorphisms) and functorialy, 
to each of the following exact sequences (where, for example, $\ol$ is the canonical function sheaf of the sphere (identified with 
the sheaf of holomorphic section of the trivial line bundle $t\times\C$)\,, $\ol(-1)$ is the tautological line bundle over the sphere (identified with the 
complex projective line), and $\mathcal{U}(-1)=\mathcal{U}\otimes\ol(-1)$\,): 
\begin{equation} \label{e:first_canonical} 
0\longrightarrow \ol(-1)\otimes H^0\bigl(\mathcal{U}(-1)\bigr)\longrightarrow\ol\otimes H^0(\mathcal{U})\longrightarrow\mathcal{U}\longrightarrow0\;, 
\end{equation} 
\begin{equation} \label{e:second_canonical}
0\longrightarrow\ol\otimes H^0\bigl(\mathcal{U}(-2)\bigr)\longrightarrow\ol(1)\otimes H^0\bigl(\mathcal{U}(-1)\bigr)\longrightarrow\mathcal{U}
\longrightarrow\ol\otimes H^1\bigl(\mathcal{U}(-2)\bigr)\longrightarrow0\;. 
\end{equation} 
\end{prop} 
\begin{proof} 
\indent 
The sequence \eqref{e:first_canonical} is the exact sequence of \cite[Proposition 3.1]{Qui-QJM98}\,. It is useful, however, to recall how it is obtained. 
Firstly, as $\mathcal{U}$ is nonnegative, the canonical morphism $\ol\otimes H^0(\mathcal{U})\longrightarrow\mathcal{U}$ is surjective, and the induced exact sequence 
of cohomology implies that its kernel is of the form $\ol(-1)\otimes V$ for some (complex) vector space $V$. Then, on tensorising 
$0\longrightarrow \ol(-1)\otimes V\longrightarrow\ol\otimes H^0(\mathcal{U})\longrightarrow\mathcal{U}\longrightarrow0$ with $\ol(-1)$ 
and passing to the cohomology exact sequence we obtain that $V=H^0\bigl(\mathcal{U}(-1)\bigr)$\,.\\ 
\indent 
To construct \eqref{e:second_canonical}\,, note that, there exists a canonical exact sequence 
\begin{equation} \label{e:for_second_canonical_1} 
0\longrightarrow\mathcal{U}_1\longrightarrow\mathcal{U}\longrightarrow\mathcal{U}_0\longrightarrow0\;, 
\end{equation} 
where $\mathcal{U}_1$ is positive and $\mathcal{U}_0$ is trivial; moreover, $\mathcal{U}_0=\ol\otimes H^1\bigl(\mathcal{U}(-2)\bigr)$\,. 
Then $\mathcal{U}_1(-1)$ is nonnegative and therefore \eqref{e:first_canonical} holds for it. As
$H^0\bigl(\mathcal{U}(-j)\bigr)=H^0\bigl(\mathcal{U}_1(-j)\bigr)$\,, for any $j\in\mathbb{N}\setminus\{0\}$\,, on tensorising with $\ol(1)$ 
the exact sequence \eqref{e:first_canonical}\,, applied to $\mathcal{U}_1(-1)$\,, we obtain 
\begin{equation*}  
0\longrightarrow \ol\otimes H^0\bigl(\mathcal{U}(-2)\bigr)\longrightarrow\ol(1)\otimes H^0\bigl(\mathcal{U}(-1)\bigr)\longrightarrow\mathcal{U}_1\longrightarrow0\;,   
\end{equation*} 
which, together with \eqref{e:for_second_canonical_1}\,, gives \eqref{e:second_canonical} (with its uniqueness and functoriality). 
\end{proof} 

\begin{rem} \label{rem:ro_from_canonical} 
Let $\mathcal{U}$ be a nonnegative holomorphic vector bundle over the sphere, endowed with a conjugation covering the antipodal map. 
Then the complexification of the corresponding linear $\r$-quaternionic structure is obtained by applying the functor $H^0$ to the morphism 
$\ol(1)\otimes H^0\bigl(\mathcal{U}(-1)\bigr)\longrightarrow\mathcal{U}$ from \eqref{e:second_canonical}\,.  
\end{rem} 

\begin{proof}[Proof of Theorem \ref{thm:basic_q-like_corresp}\,]  
If we dualize (ii) we obtain (iii) applied to $U^*$. Furthermore, if (ii)/(iii) holds then there exists a unique linear quaternionic-like structure on $U$, 
with respect to which $\mathcal{U}_{\pm}$ is the holomorphic vector bundle of $U_{\pm}$\,, and $\psi_{\pm}$ is quaternionic-like linear and 
induces the identity morphisms on $\mathcal{U}_{\pm}$.\\ 
\indent 
To prove the converse, for simplicity, we work in the category of complex vector spaces. Thus, we shall ignore the conjugations, and, in particular, 
we shall denote in the same way $U$ and its complexification. Then endow $U$ with a linear quaternionic-like structure. 
By Proposition \ref{prop:basic_q-like_corresp}\,, this means that we have an exact sequence of holomorphic vector bundles over the sphere, as follows 
\begin{equation} \label{e:exact_seq_for_q-like}  
0\longrightarrow\mathcal{U}_-\longrightarrow\ol\otimes U\longrightarrow\mathcal{U}_+\longrightarrow0\;. 
\end{equation} 
\indent 
Let $U_+=H^0(\mathcal{U}_+)$ and let $\psi_+:U\to U_+$ the induced linear map. By duality, we, also, obtain $\psi_-:U_-\to U$. 
Further, on tensorising \eqref{e:exact_seq_for_q-like} with $\ol(-1)$ and by passing to the cohomology sequence we obtain 
\begin{equation} \label{e:E+=E-} 
H^0\bigl(\mathcal{U}_+(-1)\bigr)=H^1\bigl(\mathcal{U}_-(-1)\bigr)=H^0\bigl(\mathcal{U}_-^*(-1)\bigr)^*\;, 
\end{equation}  
where the second equality is given by the Serre duality. Together with Remark \ref{rem:ro_from_canonical}\,, this implies that $E_+=E_-$ 
(as (complex-)quaternionic vector spaces)\,, where $(E_+,\r_+)$ and $(E_-,\r_-^*)$ are the linear $\r$-quaternionic and the linear $\r^*\!$-quaternionic 
structures of $U_+$ and $U_-$\,, respectively.\\ 
\indent 
Now, we claim to have the following commutative diagram, in which all of the vertical and horizontal sequences are exact. 

\begin{displaymath} 
\xymatrix{
                 &          0 \ar[d]                                                                &   0  \ar[d]                                                              &                  &  \\
   0 \ar[r]   &   \ol\otimes H^0(\mathcal{U}_-) \ar[d] \ar[r]^{=}     &   \ol\otimes H^0(\mathcal{U}_-) \ar[d] \ar[r]      &   0 \ar[d]    &  \\
   0 \ar[r]   &   \mathcal{U}_- \ar[d] \ar[r]                                        &   \ol\otimes U \ar[d] \ar[r]          & \mathcal{U}_+ \ar[d]^{=} \ar[r]   &   0  \\
   0 \ar[r]   &   \ol(-1)\otimes H^0\bigl(\mathcal{U}_+(-1)\bigr) \ar[d] \ar[r]   &   \ol\otimes U_+ \ar[d] \ar[r]          & \mathcal{U}_+ \ar[d] \ar[r]   &   0   \\
   0 \ar[r]   &   \ol\otimes H^1(\mathcal{U}_-) \ar[d] \ar[r]^{=}       &   \ol\otimes H^1(\mathcal{U}_-) \ar[d] \ar[r]      &   0           &   \\ 
                 &                0                                                                      &                      0                                                      &                 &   }
\end{displaymath}  

\indent 
Indeed, the second horizontal sequence is just \eqref{e:exact_seq_for_q-like}\,, the third horizontal sequence is \eqref{e:first_canonical} 
applied to $\mathcal{U}_+$\,, whilst the morphisms, corresponding to vertical arrows, between these two sequences are induced by $\psi_+$\,.\\ 
\indent 
Further, the second vertical sequence is the cohomology sequence of \eqref{e:exact_seq_for_q-like}\,. By using the fact that $\psi_+$ is injective 
if and only if $\mathcal{U}_-$ is negative we, also, obtain the first horizontal sequence, and by passing to cokernels we may complete the diagram 
with its fourth horizontal sequence.\\ 
\indent 
Now, by using the Serre duality, and \eqref{e:E+=E-} we deduce that the first vertical sequence of the diagram is just the dual of \eqref{e:second_canonical} 
applied to $\mathcal{U}_-^*$\,. Thus, if we dualize the diagram, pass to the cohomology sequence of the horizontal sequences, and then dualize, again 
(recall, also, Remark \ref{rem:ro_from_canonical}\,), we deduce the following facts:\\ 
\indent 
\quad(a) $\psi_+\circ\psi_-=\r_+\circ\r_-^*$\,,\\ 
\indent 
\quad(b) $\r_-^*$ maps ${\rm ker}\,\psi_-$ isomorphically onto ${\rm ker}\r_+$\,,\\ 
\indent 
\quad(c) $\psi_-$ maps ${\rm ker}\r_-^*$ isomorphically onto ${\rm ker}\,\psi_+$\,,\\ 
\indent 
\quad(d) $\r_+$ maps ${\rm coker}\r_-^*$ isomorphically onto ${\rm coker}\,\psi_+$\,,\\ 
\indent 
\quad(c) $\psi_+$ maps ${\rm coker}\,\psi_-$ isomorphically onto ${\rm coker}\r_+$\,.\\ 
\indent 
Finally, properties (b) and (c)\,, above, quickly complete the proof of statement (iii)\,, and the proof of the theorem is complete.  
\end{proof}

\newpage 

\section{Quaternionic-like manifolds and heaven spaces} 

\indent 
We are interested in homogeneous complex manifolds $Z$ endowed with a (holomorphically) embedded (Riemann) sphere $t\subseteq Z$ 
with nonnegative normal bundle. For the corresponding differential geometry to be \emph{real} we have to assume that $Z$ is endowed with an 
antiholomorphic involution $\t$ without fixed points, such that $\t(t)=t$\,, and the complex Lie group $G$ acting holomorphically 
and transitively on $Z$ is the complexification of the (real) Lie subgroup of $G$ formed of those elements whose induced 
holomorphic diffeomorphisms on $Z$ commute with $\t$.\\ 
\indent 
Along the way, we shall see that such pairs $(Z,t)$ are abundant. In this section, we answer to the following question: 
\emph{what geometric structure is induced on the (local) orbits of the $\r$-quaternionic manifold 
whose twistor space is $Z$ and for which $t$ is a twistor sphere}? The importance of this problem 
is strengthened by the proof of the following simple proposition, and its consequences.  

\begin{prop} \label{prop:homog_nonnegative} 
If $Z$ is a homogeneous complex manifold then any sphere embedded in $Z$ has nonnegative normal bundle. 
\end{prop} 
\begin{proof} 
Let $t\subseteq Z$ be an embedded sphere, and let $G$ be a complex Lie group acting holomorphically on $Z$. 
Denote by $G_t$ the closed subgroup of $G$ formed of the elements preserving $t$\,. Thus, 
each point $aG_t\in G/G_t$ corresponds to the sphere $at\subseteq Z$.\\ 
\indent 
 The map $G\times t\to Z\times G/G_t$\,, $(a,z)\mapsto(az,at)$\,, induces an embedding of $G\times_{G_t}t$ into $Z\times G/G_t$ 
 whose image is formed of those pairs $(w,s)\in Z\times G/G_t$ such that $w\in s$\,.\\ 
\indent 
Suppose, now, that $G$ acts transitively on $Z$. Then $G\times_{G_t}t\to Z$\,, $[a,z]\mapsto az$\,, is a surjective holomorphic submersion.  
Consequently, the normal bundle of $t$ in $Z$ is a quotient of the trivial bundle over $t$ with typical fibre the holomorphic tangent space 
to $G/G_t$ at its `origin'. Hence, the normal bundle of $t$ in $Z$ is nonnegative, and the proof is complete. 
\end{proof} 

\begin{rem} 
1) It is easy to see that any (generalized) flag manifold admits an embedded sphere. Thus, Proposition \ref{prop:homog_nonnegative} and \cite{Pan-qgfs}\,, 
implies that any flag manifold is the twistor space of a (nonunique) $\r$-quaternionic manifold.\\ 
\indent 
2) With the same notations as in the proof of Proposition \ref{prop:homog_nonnegative}\,, and assuming that $G$ acts transitively on $Z$, 
we have that $G$ acts transitively on $G\times_{G_t}t\subseteq Z\times G/G_t$ if and only if $G_t$ acts transitively on $t$\,. 
If this holds, then $G\times_{G_t}t=G/(K\cap G_t)$\,, where $K\subseteq G$ is the isotropy group 
of the action of $G$ on $Z$ at some point of $t$\,. 
\end{rem} 

\indent 
Now, to answer to the above mentioned question we have to introduce the following notion. 

\begin{defn} \label{defn:almost_q-like_str} 
Let $M$ be a manifold and let $k\in\mathbb{N}$\,, $k\leq\dim M$. An \emph{almost quaternionic-like structure} on $M$ is a pair $(Y,t)$\,, 
where $Y$ is an oriented sphere bundle over $M$, and $t:Y\to{\rm Gr}_k\bigl(T^{\C\!}M\bigr)$ is a bundle morphism such that 
$t_x$ is a linear quaternionic-like structure on $T_xM$, for any $x\in M$.\\ 
\indent 
An \emph{almost quaternionic-like manifold} is a manifold endowed with an almost qua\-ter\-ni\-o\-nic-like structure.  
\end{defn}  

\indent 
Let $(M,Y,t)$ be an almost quaternionic-like manifold. We denote by $\mathcal{T}^{\pm}M$ the complex vector bundles obtained as follows 
(cf.\ Proposition \ref{prop:basic_q-like_corresp}\,). Firstly, $\mathcal{T}^-M$ is the pull-back by $t$ of the tautological bundle over 
${\rm Gr}_k\bigl(T^{\C\!}M\bigr)$ and then $\mathcal{T}^+M$ 
is the cokernel of the inclusion $\iota:\mathcal{T}^-M\to\p^*\bigl(T^{\C\!}M\bigr)$\,, where $\p:Y\to M$ is the projection.\\ 
\indent 
From \cite[Theorem 2.3]{KodSpe-I_II}\,, \cite[Theorem 9]{KodSpe-III}\,, and Theorem \ref{thm:basic_q-like_corresp}\,, we deduce 
the existence of vector bundles $T^{\pm}M$ and $E\,(=E_{\pm})$\,, and morphisms $\psi_{\pm}$ and $\r_{\pm}$ 
which determine the given almost quaternionic-like structure.\\ 
\indent 
Let $c:\mathcal{T}^-M\to T^{\C\!}Y$ be a morphism of complex vector bundles over $Y$ which intertwines the conjugations, 
and such that $\widetilde{\dif\!\p}\circ c=\iota$\,, 
where  $\widetilde{\dif\!\p}:T^{\C\!}Y\to\p^*\bigl(T^{\C\!}M\bigr)$ is the morphism of complex vector bundles induced by $\dif\!\p$\,.   
We shall denote $\Cal=({\rm ker}\dif\!\p)^{0,1}\oplus({\rm im}\,c)$\,. 

\begin{defn} \label{defn:q-like_str} 
The almost quaternionic-like structure $(Y,t)$ is \emph{integrable, with respect to $c$\,,} if the space of smooth sections of $\Cal$ 
is closed under the usual bracket. Then $(Y,t,c)$ is a \emph{quaternionic-like structure} and $(M,Y,t,c)$ is a \emph{quaternionic-like manifold}.\\ 
\indent 
Suppose, further, that there exists a surjective submersion $\phi:Y\to Z$, where $Z$ is endowed with a CR structure $\Cal_Z$ such that:\\ 
\indent 
\quad(1) $\dif\!\phi(\Cal)=\Cal_Z$\,,\\ 
\indent 
\quad(2) $({\rm ker}\dif\!\phi)^{\C\!}=\Cal\cap\overline{\Cal}$\,,\\ 
\indent 
\quad(3) $\phi$ restricted to each fibre of $\p$ is injective,\\ 
\indent 
\quad(4) the antipodal map on $Y$ descends to an (anti-CR) involution $\t$ on $Z$.\\  
\indent 
Then the CR manifold $(Z,\Cal_Z)$ is the \emph{twistor space} of $(M,Y,t,c)$\,. 
\end{defn} 

\indent 
Note that, in Definition \ref{defn:q-like_str}\,, as $\Cal$ is integrable, a sufficient condition for $Z$\,, with the required properties, to exist, is 
that $\Cal\cap\overline{\Cal}$ be the complexification of the tangent bundle of a foliation given by a submersion $\phi$ with connected fibres.\\ 
\indent 
Also, if, in Definition \ref{defn:q-like_str}\,, $(Y,t)$ is (given by) an almost CR quaternionic structure then, by applying \cite[Theorem 2.1]{Pan-twistor_(co-)cr_q}\,, 
we retrieve the notion of CR quaternionic manifold of \cite{fq}\,. Moreover, we shall, next, see that, assuming real analyticity, the notion of `heaven space' 
appears naturally in the current, more general, setting. In particular, if $(Y,t)$ is an almost $\r$-quaternionic structure 
then, up to a complexification, we retrieve the $\r$-quaternionic manifolds of \cite{Pan-qgfs}\,.\\ 
\indent 
In the next result, by a `Hausdorff manifold' we mean a manifold whose underlying topological space is Hausdorff (not necessarily paracompact).       

\begin{thm} \label{thm:heaven_space_for_q-like} 
Let $(M,Y,t,c)$ be a real analytic quaternionic-like manifold with twistor space $(Z,\Cal_Z)$\,, given by the surjective submersion $\phi$\,. 
Then there exists a (germ) unique Hausdorff quaternionic-like manifold $\bigl(M^+,Y^+,t^+,c^+\bigr)$\,, with twistor space 
$\bigl(Z^+,\Cal^+\bigr)$\,, given by $\phi^+$, and a unique quaternionic-like map $(\psi,\Psi)$ from $M$ to $M^+$ 
such that the following assertions are satisfied:\\ 
\indent 
{\rm (i)} $\bigl(Y^+,t^+\bigr)$ is an almost $\r$-quaternionic structure (equivalently, $TM^+=T^+M^+$),\\ 
\indent 
{\rm (ii)} $\psi^*\bigl(TM^+\bigr)=T^+M$ and the morphism from $TM$ to $T^+M$ induced by $\dif\!\psi$ is $\psi_+$\,,\\ 
\indent 
{\rm (iii)} $\Psi$ induces a CR embedding $\psi_Z:Z\to Z^+$ such that $\psi_Z\circ\phi=\phi^+\circ\Psi$.\\ 
\noindent 
(The `germ uniqueness', here, means that, up to diffeomorphisms, the germ of $M^+$ (as a quaternionic-like manifold) along the image of $\psi$ is uniquely determined.)  
\end{thm} 
\begin{proof} 
By \cite{AndFre}\,, we may assume that $Z$ is a CR submanifold of a complex manifold $\widetilde{Z}$ such that the inclusion 
$T^{\C\!}Z\subseteq T^{\C\!}\widetilde{Z}$ induces an isomorphism between the following two complex vector bundles:\\ 
\indent 
\quad(1) the restriction to $Z$ of the holomorphic tangent bundle of $\widetilde{Z}$, and\\ 
\indent 
\quad(2) the quotient of $T^{\C\!}Z$ through $\Cal_Z$\,.\\ 
(In fact, we may, essentially, assume that $Z$ admits a complexification such that $\Cal_Z$ is the restriction to $Z$ of a holomorphic foliation whose leaf space is 
$\widetilde{Z}$.)\\ 
\indent 
Moreover, we may assume that $\t$ is the restriction to $Z$ of an antiholomorphic involution $\widetilde{\t}$ on $\widetilde{Z}$, without fixed points.\\ 
\indent 
Let $\widetilde{Y}$ be the complexification of $Y$ as a sphere bundle over $M$. Then $\widetilde{Y}$ is a sphere bundle, whose projection we denote by 
$\widetilde{\p}$\,, over a complexification $\widetilde{M}$ of $M$\,; in particular, $\widetilde{\p}|_Y=\p$\,. 
(Note that, $\widetilde{Y}$ is not a complexification of $Y$, as a real analytic manifold, but the `complexification' of $Y$ endowed with the CR structure 
given by the antiholomorphic tangent bundles of its fibres.) We may assume that the antipodal map extends to an involution 
without fixed points on $\widetilde{Y}$ which covers, through $\widetilde{\p}$\,, the conjugation of $\widetilde{M}$ (whose fixed point set is $M$).\\ 
\indent 
By passing, if necessary, to open neighbourhoods of $\widetilde{Z}$ and $\widetilde{M}$ we may assume that $\phi$ extends 
to a holomorphic surjective submersion $\widetilde{\phi}:\widetilde{Y}\to\widetilde{Z}$ which intertwines the conjugations. 
We claim that we may, further, assume that the restriction of $\widetilde{\phi}$ to each fibre of $\widetilde{\p}$ is an embedding. 
For this, we need the deformation theory of compact complex submanifolds (see \cite{Nam-83}\,, \cite{Ro-LeB_nonrealiz}\,; cf.\ \cite{Kod}\,). 
Let $S$ be the moduli space of compact complex one-dimensional submanifolds of $\widetilde{Z}$. This is a Hausdorff reduced complex analytic space 
endowed with a conjugation induced by $\widetilde{\t}$\,. Moreover, as $M$ is a parameter space for a real-analytic family of embedded spheres 
in $\widetilde{Z}$ there exists a unique map $\psi:M\to S$ corresponding to $\phi$ and the embedding $Z\subseteq\widetilde{Z}$\,. 
Now, the isomorphism between (1) and (2)\,, above, implies that, for any $x\in M$, the normal bundle of the sphere $\phi\bigl(\p^{-1}(x)\bigr)$ in $\widetilde{Z}$ 
is isomorphic to $\mathcal{T}^+_xM$ which, because is nonnegative, has the property that $H^1\bigl(\mathcal{T}^+_xM\bigr)=0$\,. 
Consequently, the image of $\psi$ is formed of smooth points of $S$\,. Moreover, we may assume that $\psi$ admits a holomorphic extension 
$\widetilde{\psi}$ over $\widetilde{M}$ corresponding to a holomorphic extension of the CR map $\phi:Y\to Z$, and whose image is formed of smooth points. 
By uniqueness this should correspond to $\widetilde{\phi}$\,, after passing, if necessary, to an open neighbourhood of $M$ in $\widetilde{M}$. 
As the image of $\widetilde{\psi}$ is formed of smooth points (corresponding to spheres embedded into $\widetilde{Z}$\,) this proves our claim.\\ 
\indent 
Finally, by applying \cite{Pan-qgfs}\,, an open neighbourhood of the image of $\psi$ is the complexification of a Hausdorff $\r$-quaternionic manifold 
$M^+$, as required, thus, completing the proof.  
\end{proof} 

\indent 
We call the $\r$-quaternionic manifold $M^+$ of Theorem \ref{thm:heaven_space_for_q-like} the \emph{heaven space} of $M$.  
 
\begin{exm} 
Any quaternionic-like vector space is a real analytic quaternionic-like manifold, in an obvious way, and, if $\dim\bigl(U^z\cap\overline{U^z}\bigr)$ 
does not depend of $z$\,, then its heaven space is given by the $\r$-quaternionic vector space $U_+$\,. 
\end{exm}  

\begin{exm} 
As already mentioned, any CR quaternionic manifold $M$ is quaternionic-like. Furthermore, assuming real analyticity, the heaven space of $M$, 
as a CR quaternionic manifold, is equal to the heaven of $M$, as a quaternionic-like manifold. 
\end{exm} 

\indent 
It is straightforward to extend Theorem \ref{thm:heaven_space_for_q-like} to obtain a result similar to \cite[Theorem 5.3]{fq}\,. 

\begin{exm} 
With the same notations as in Proposition \ref{prop:homog_nonnegative}\,, the homogeneous complex manifold $G/G_t$ 
is a complex-quaternionic-like manifold. Consequently, if $M$ is a $\r$-quaternionic manifold whose twistor space is homogeneous 
then each orbit, of the induced local action on $M$, is endowed with a real-analytic quaternionic-like structure whose heaven space is $M$. 
\end{exm} 

\indent 
For simplicity, no functoriality statements have been made in this section. The reader could easily obtain these 
by adapting the notion of `twistorial map' of \cite{PanWoo-sd}\,, \cite{LouPan-II}\,.

\section{A construction of homogeneous twistor spaces} \label{section:constr_hqo} 

\indent 
We shall give and characterise a construction of homogeneous complex manifolds endowed with embedded Riemann spheres 
(with nonnegative normal bundles), and we shall describe several classes of $\r$-quaternionic manifolds whose twistor spaces 
are obtained by this construction.\\ 
\indent 
The ingredients of the construction are:\\ 
\indent 
\quad(1) a complex Lie algebra $\mathfrak{g}$\,,\\ 
\indent 
\quad(2) a (complex) representation $\s:\mathfrak{g}\to\mathfrak{gl}(E)$\,,\\ 
\indent 
\quad(3) an injective morphism of complex Lie algebras $\t:\mathfrak{sl}(2,\C)\to\mathfrak{g}$\,,\\ 
\indent 
\quad(4) a complex vector subspace $U\subseteq E$ such that $A\mapsto (\s\circ\t)(A)|_U$ 
is an irreducible nontrivial representation of $\mathfrak{sl}(2,\C)$ on $U$. 

\begin{rem} 
1) Condition (3)\,, above, implies that $\mathfrak{g}$ is not solvable (see \cite[\S\,6, n$^o$\,8, cor.\,1]{Bou-Lie_I}\,). Conversely, for any 
Lie algebra $\mathfrak{g}$ which is not solvable we can find $\s$, $\t$, $U$ as in (1)-(4)\,, above. Indeed, if $\mathfrak{g}$\,, $\s$\,, $\t$ 
are as in (1)-(3)\,, then there exists $U$ such that (4) is satisfied if and only if $\s\circ\t$ is nontrivial.\\ 
\indent 
2) To obtain \emph{real} $\r$-quaternionic manifolds, in (1)-(3)\,, we use complexifications (with $\mathfrak{su}(2)$\,, in (3)\,), and, 
in (4)\,, we assume $U$ invariant under the conjugation of $E$. 
\end{rem} 

\indent 
To simplify the exposition, we shall call a quadruple $(\mathfrak{g},\s,\t,U)$ as in (1)-(4)\,, above, \emph{good}. 
With these quadruples we can form a category whose morphisms are pairs 
$(\phi,\psi):(\mathfrak{g},\s_{\mathfrak{g}},\t_{\mathfrak{g}},U_{\mathfrak{g}})\to(\mathfrak{h},\s_{\mathfrak{h}},\t_{\mathfrak{h}},U_{\mathfrak{h}})$\,, 
where $\phi:\mathfrak{g}\to\mathfrak{h}$ is a morphism of Lie algebras such that $\phi\circ\t_{\mathfrak{g}}=\t_{\mathfrak{h}}$, 
and, if $E_{\mathfrak{g}}$ and $E_{\mathfrak{h}}$ are the representation spaces of $\s_{\mathfrak{g}}$ and $\s_{\mathfrak{h}}$, respectively, then 
$\psi:E_{\mathfrak{g}}\to E_{\mathfrak{h}}$ is a $\phi$-equivariant complex linear map 
such that $\psi(U_{\mathfrak{g}})=U_{\mathfrak{h}}$ (in particular, $U_{\mathfrak{g}}$ and $U_{\mathfrak{h}}$ 
are isomorphic as representation spaces of $\mathfrak{sl}(2,\C)$\,).  

\begin{prop} 
Let $(\mathfrak{g},\s,\t,U)$ be a good quadruple, and let $t\subseteq PU$ be the Veronese curve  formed of all\/ ${\rm ker}\bigl((\s\circ\t)(A)|_U\bigr)$\,,  
with $A\in\mathfrak{sl}(2,\C)$ nilpotent.\\ 
\indent  
Let $Z$ be the orbit containing\/ $t$\,, of the induced action of $G$ on $PE$, where $G$ is the simply-connected complex Lie group 
whose Lie algebra is $\mathfrak{g}$\,.\\ 
\indent 
Then the normal bundle of\/ $t$ in $Z$ is nonnegative, and $(\mathfrak{g},\s,\t,U)\mapsto(Z,t)$ is functorial. 
\end{prop} 
\begin{proof} 
This follows quickly from Proposition \ref{prop:homog_nonnegative}\,.  
\end{proof} 

\indent 
To state the next result, we recall two standard notions of algebraic geometry (see \cite{GriHar}\,). 
Let $N$ be a complex manifold (not necessarily compact) and let $\mathcal{L}$ be a holomorphic line bundle over $N$. Also, let $E$ be a finite dimensional 
complex vector subspace of the space of holomorphic sections of $\mathcal{L}$\,. We say that $\mathcal{L}$ is \emph{very ample, with respect to $E$}, 
if on associating to any $x\in N$ the subspace of $E$ formed of those sections that vanish at $x$ we obtain a well-defined injective immersion into $PE^*$.\\ 
\indent  
On the other hand, let $N$ be a compact complex submanifold of a complex projective space $PE$, and let $\mathcal{L}$ be the dual of the tautological 
line bundle over $PE$. Then $N$ is \emph{normal} if the restriction map from $E^*$ to the space of holomorphic sections of $\mathcal{L}|_N$ is surjective. 
 
\begin{thm}  
There exists a natural (surjective) correspondence between {\rm (i)} good quadruples, and {\rm (ii)} triples $(Z,t,\mathcal{L})$\,, with $Z=G/K$ a homogeneous complex manifold, 
$t\subseteq Z$ an embedded Riemann sphere, and $\mathcal{L}$ a holomorphic line bundle over $Z$, such that:\\ 
\indent 
{\rm (a)} the morphism $G_t\to{\rm PGL}(2,\C)$\,, $a\mapsto a|_t$ is surjective, where $G_t\subseteq G$ is the subgroup preserving $t$\,,\\ 
\indent 
{\rm (b)} the action of $G$ on $Z$ lifts to a holomorphic action of $G$ on $\mathcal{L}$ by automorphisms,\\ 
\indent 
{\rm (c)} $\mathcal{L}$ is very ample, with respect to a $G$-invariant subspace $F$ of the space of holomorphic sections of $\mathcal{L}$\,,\\ 
\indent 
{\rm (d)} the induced embedding of $t$ into $PF^*$ is normal. 
\end{thm} 
\begin{proof} 
It is sufficient to show how to pass from (ii) to (i)\,. If $(Z,t,\mathcal{L})$ is a triple as in (ii)\,, with $Z=G/K$, let $\mathfrak{g}$ be the Lie algebra 
of $G$, let $E=F^*$, and let $U$ be the dual of the space of holomorphic sections of $\mathcal{L}|_t$\,. 
From (d)\,, we obtain that $U\subseteq E$, and, from (b)\,, (c)\,, we obtain a representation $\s$ of $\mathfrak{g}$ on $E$\,.\\ 
\indent 
If $\mathfrak{g}_t$ is the Lie algebra of $G_t$ then, by (a)\,, we have a surjective morphism of complex Lie algebras 
$\mu:\mathfrak{g}_t\to\mathfrak{sl}(2,\C)$\,. As $\mathfrak{sl}(2,\C)$ 
is simple, there exists a Lie algebras section $\t$ of $\mu$ (see \cite[\S\,6, n$^o$\,8, cor.\,3]{Bou-Lie_I}\,).\\ 
\indent 
Finally, (d) implies that the representation of $\mathfrak{sl}(2,\C)$\,, induced on $U$, is irreducible and of dimension $d+1$\,, 
where $d$ is the degree of $t\subseteq PE$\,. Thus, $(\mathfrak{g},\s,\t,U)$ is good, and the proof is complete. 
\end{proof} 

\begin{exm}[compare \cite{Pan-qgfs}\,] 
The irreducible representations of $\mathfrak{sl}(2,\C)$ are given by injective morphisms of Lie algebras $\t_k:\mathfrak{sl}(2,\C)\to\mathfrak{sl}(U_k)$\,, 
where $\dim_{\C\!}U_k=k+1$\,, $(k\in\mathbb{N})$\,; in particular, $\mathfrak{sl}(2,\C)=\mathfrak{sl}(U_1)$\,.\\ 
\indent 
If $k\geq1$ then $\bigl(\mathfrak{sl}(U_k),{\rm Id}_{\mathfrak{sl}(U_k)},\t_k,U_k\bigr)$ is good, and the resulting complex homogeneous manifold is $Z=PU_k$\,, 
the twistor space of the space $M_k$ of Veronese curves of degree $k$\,. Note that, $M_k$ is the quotient of ${\rm PGL}(U_k)$ through ${\rm PGL}(U_1)$ 
and therefore is homogeneous. Thus, up to the integrability, the geometric structure of $M_k$ is determined by the type of the normal bundle of 
$t\subseteq PU_k$ as a holomorphic vector bundle endowed with an action of ${\rm PGL}(U_1)$\,. 
It is fairly well-known that this normal bundle is $U_{k-2}\otimes\ol(k+2)$\,, 
where $U_{-1}=\{0\}$ (if $k$ is odd, ${\rm PGL}(U_1)$ acts on $U_{k-2}\otimes\ol(k+2)$\,, although it does not act on its factors).  
Consequently, the geometry of $M_k$ is underlied by the canonical representation of ${\rm PGL}(U_1)$ on $U_{k-2}\otimes U_{k+2}$\,. 
\end{exm} 

\begin{exm}[compare \cite{fq_2}\,] 
Let $E$ be a (real) Euclidean space and let $U\subseteq E$ be a three-dimensional vector subspace. By using the Euclidean structure, we obtain 
an injective morphism of Lie algebras $\t:\mathfrak{so}(U)\to\mathfrak{so}(E)$\,.\\ 
\indent  
Then (the complexification of) $\bigl(\mathfrak{so}(E),{\rm Id}_{{\mathfrak{so}(E)}},\t,U\bigr)$ is good, and the resulting complex homogeneous space 
is the (isotropic) hyperquadric in $PE^{\C}$. This is the twistor space of the Grassmannian of oriented three-dimensional subspaces of $E$. 
\end{exm} 

\begin{exm}[compare \cite{fq_2}\,]  
Let $(V,\o)$ be a complex symplectic vector space, $\dim_{\C}V\geq4$\,, and let $p_1$ and $p_2$ be two-dimensional complex vector subspaces 
of $V$ such that $V_1=p_1+p_2$ is nondegenerated (and four-dimensional), with respect to $\o$\,, and $\o|_{p_1}=0$\,, $\o|_{p_2}=0$\,. Alternatively, 
$p_1$ and $p_2$ are characterised by the facts that they are of dimension $2$\,, $\o|_{p_1}=0$ and $\o$ induces an isomorphism between $p_2$ and $p_1^*$.\\ 
\indent 
Let $U_1=({\rm ker}\,\o)\cap\Lambda^2V_1$\,, and, note that, $\Lambda^2p_1$ and $\Lambda^2p_2$ are subspaces of $U_1$\,. 
Moreover, $\Lambda^2p_1+\Lambda^2p_2$ is a two-dimensional subspace of $U_1$ which is nondegenerated, 
with respect to the (complex-)conformal structure, given by the wedge product. 
Therefore the orthogonal complement $U\subseteq U_1$ of $\Lambda^2p_1+\Lambda^2p_2$\,, 
with respect to this conformal structure, is nondegenerated and three-dimensional. Consequently, we have 
$\mathfrak{sl}(2,\C)=\mathfrak{so}(U)\subseteq\mathfrak{so}(U_1)=\mathfrak{sp}(V_1)\subseteq\mathfrak{sp}(V)$\,; 
denote by $\t:\mathfrak{sl}(2,\C)\to\mathfrak{sp}(V)$ the obtained injective morphism of Lie algebras.\\ 
\indent 
On denoting by $\s$ the respresentation of $\mathfrak{sp}(V)$ on $\Lambda^2V$ we obtain the good quadruple 
$\bigl(\mathfrak{sp}(V),\s,\t,U\bigr)$\,. The obtained homogeneous complex manifold is the complex hypersurface  
$Z\subseteq{\rm Gr}_2(V)$ formed of the two-dimensional complex subspaces which are isotropic with respect to $\o$\,.\\ 
\indent 
The corresponding $\r$-quaternionic manifold is the submanifold $M\subseteq Z\times Z$ formed of those $(q_1,q_2)\in Z\times Z$ 
such that $q_1+q_2$ is nondegenerated, with respect to $\o$\,. Consequently, $M$ is the quotient of ${\rm Sp}(V)$ through 
$H={\rm GL}(p_1)\times{\rm Sp}(V_2)$\,, where $V_2=(p_1\oplus p_2)^{\perp_{\o}}$\,, and to embed $H$ into ${\rm Sp}(V)$ 
we have used the complex linear isomorphism $V=(p_1\oplus p_1^*)\oplus V_2$. Then $M={\rm Sp}(V)/H$ is a reductive homogeneous space, 
and the action of $H$ on $\mathfrak{sp}(V)/\mathfrak{h}$\,, where $\mathfrak{h}$ is the Lie algebra of $H$, 
is induced through the complex linear isomorphism 
$\mathfrak{sp}(V)/\mathfrak{h}=\bigl(\odot^2p_1\bigr)\oplus\bigl(\odot^2p_1^*\bigr)\oplus\bigl((p_1\oplus p_1^*)\otimes V_2\bigr)$\,, 
where $\odot$ denotes the symmetric product. It follows that the normal bundle of each twistor sphere is $2\ol(2)\oplus2(\dim_{\C\!}V-4)\ol(1)$\,.\\ 
\indent 
To make this construction \emph{real}, think of $V$ as a real vector space and let $I$ be the given linear complex structure on it. 
Then endow $V$ with another linear complex structure $J$ such that $JI=-IJ$ and $J^*\o=\overline{\o}$\,. 
Now, the map $M\to M$, $(q_1,q_2)\mapsto(Jq_2,Jq_1)$\,, is a conjugation (that is, an involutive antiholomorphic diffeomorphism) 
whose fixed point set is the image of the real analytic map $Z\to M$\,, $p\mapsto(p,Jp)$\,; thus, $M=Z^{\C\!}$. 
This way, $Z$ is an $f$-quaternionic manifold and its twistor space is $Z$ itself.  
\end{exm} 

\indent 
For the next example, recall \cite{Kos} that if $\mathfrak{g}$ is a complex semisimple Lie algebra then there 
is a natural correspondence between nilpotent adjoint orbits and conjugation classes of Lie subalgebras of $\mathfrak{g}$ 
isomorphic to $\mathfrak{sl}(2,\C)$\,. The involved association is the following: if $Y\in\mathfrak{sl}(2,\C)\setminus\{0\}$ 
is nilpotent (equivalently, $Y^2=0$, in the canonical representation) and $\mathfrak{sl}(2,\C)$ is a Lie subalgebra of $\mathfrak{g}$ 
then $Y\in\mathfrak{g}$ is nilpotent, and, obviously, the adjoint orbit of $Y$ depends only of the conjugation class of 
the embedding $\mathfrak{sl}(2,\C)\subseteq\mathfrak{g}$\,. Conversely, any nilpotent $Y\in\mathfrak{g}\setminus\{0\}$ 
is contained in a Lie subalgebra of $\mathfrak{g}$ isomorphic to $\mathfrak{sl}(2,\C)$ (this fact is, usually, called the 
Jacobson--Morozov Theorem), and, moreover, any two such subalgebras are conjugated.  

\begin{exm}[compare \cite{Swa-99}\,] \label{exm:nil_orbits} 
Let $\mathfrak{g}$ be a complex semisimple Lie algebra, and let ${\rm ad}:\mathfrak{g}\to\mathfrak{gl}(\mathfrak{g})$ 
be its adjoint representation. Then, for any Lie algebra embedding $\t:\mathfrak{sl}(2,\C)\to\mathfrak{g}$\,, the quadruple 
$(\mathfrak{g},{\rm ad},\t,{\rm im}\,\t)$ is good.\\ 
\indent 
Let $(Z,t)$ be the resulting homogeneous complex manifold, endowed with an embedded Riemann sphere. 
To describe the normal bundle of $t$ in $Z$, for any $j\in\mathbb{N}$\,, let $a_j\in\mathbb{N}$ be the multiplicity with which 
the irreducible component $U_j$ appears in the decomposition of the represention of $\mathfrak{sl}(2,\C)$ on $\mathfrak{g}$\,, induced by ${\rm ad}$ and $\t$\,. 
Then the normal bundle of $t$ in $Z$ is $\bigl(-2+\sum_{j\in\mathbb{N}}ja_j\bigr)\ol(1)$ (note that, the previous sum contains only finitely 
many nonzero terms).\\ 
\indent 
It is known that any compact real form $\mathfrak{g}_{\R}$ of $\mathfrak{g}$ - which we assume such that 
$\mathfrak{g}_{\R}\cap{\rm im}\,\t=\mathfrak{su}(2)$ - determines a conjugation, with respect to which, 
$Z$ is the twistor space of a quaternionic-K\"ahler manifold $M$, and $t$ is a real twistor sphere.\\ 
\indent 
Then $M$ is a co-CR quaternionic submanifold of ${\rm Gr}_3^+\bigl(\mathfrak{g}_{\R}\bigr)$\,. Furthermore, 
the conjugation class $N$ of the embedding $\mathfrak{su}(2)\subseteq\mathfrak{g}_{\R}$ is a CR quaternionic submanifold of $M$\,; 
moreover, $M$ is the heaven space of $N$. 
On the other hand, if there exists $j\geq3$ such that $a_j\neq0$\,, then $N$ is not a CR quaternionic submanifold of ${\rm Gr}_3^+\bigl(\mathfrak{g}_{\R}\bigr)$\,.\\ 
\indent 
Finally, it is known that the following assertions are equivalent:\\ 
\indent 
\quad(a) $M$ is compact,\\ 
\indent 
\quad(b) $M=N$,\\ 
\indent 
\quad(c) $M$ is a Wolf space (that is, $\mathfrak{g}$ is simple and $Z$ is the projectivisation of the adjoint orbit of a highest root vector, 
with respect to a Cartan subalgebra of $\mathfrak{g}$\,). 
\end{exm}


\begin{thebibliography}{10} 

\bibitem{AndFre}
A.~Andreotti, G.~A.~Fredricks, Embeddability of real analytic Cauchy-Riemann manifolds,
\textit{Ann. Scuola Norm. Sup. Pisa Cl. Sci. (4)}, {\bf 6} (1979) 285--304. 

\bibitem{Bou-Lie_I} 
N.~Bourbaki, \'El\'ements de math\'ematique. Groupes et alg\`ebres de Lie. Chapitre I: Alg\`ebres de Lie. Seconde \'edition. 
Actualit\'es Scientifiques et Industrielles, No. 1285, Hermann, Paris, 1971.  

\bibitem{GriHar}
P.~A.~Griffiths, J.~Harris, \textit{Principles of algebraic geometry}, Wiley Classics Library,
John Wiley \& Sons, Inc., New York, 1978. 

\bibitem{Kod}
K.~Kodaira, A theorem of completeness of characteristic systems for analytic families of
compact submanifolds of complex manifolds,
\textit{Ann. of Math. (2)}, {\bf 75} (1962) 146--162. 

\bibitem{KodSpe-I_II} 
K.~Kodaira, D.~C.~Spencer, On deformations of complex analytic structures. I, II,  
\textit{Ann. of Math. (2)}, {\bf 67} (1958) 328--466. 

\bibitem{KodSpe-III} 
K.~Kodaira, D.~C.~Spencer, On deformations of complex analytic structures. III. Stability theorems for complex structures, 
\textit{Ann. of Math. (2)}, {\bf 71} (1960) 43--76. 

\bibitem{Kos} 
B.~Kostant, The principal three-dimensional subgroup and the Betti numbers of a complex simple Lie group, 
\textit{Amer. J. Math.}, {\bf 81} (1959) 973--1032. 

\bibitem{LouPan-II} 
E.~Loubeau, R.~Pantilie, Harmonic morphisms between Weyl spaces and twistorial maps II, 
\textit{Ann. Inst. Fourier (Grenoble)}, {\bf 60} (2010) 433--453. 

\bibitem{fq}
S.~Marchiafava, L.~Ornea, R.~Pantilie, Twistor Theory for CR quaternionic manifolds and related structures,
\textit{Monatsh. Math.}, {\bf 167} (2012) 531--545.  

\bibitem{fq_2}
S.~Marchiafava, R.~Pantilie, Twistor Theory for co-CR quaternionic manifolds and related structures,
\textit{Israel J. Math.}, {\bf 195} (2013) 347--371.  

\bibitem{Nam-83} 
M.~Namba, Deformations of compact complex manifolds and some related topics, 
in \textit{Recent progress of algebraic geometry in Japan}, M.~Nagata (editor), North-Holland Math. Stud., 73, North-Holland, Amsterdam, 1983, 51--81. 

\bibitem{vq} 
R.~Pantilie, On the classification of the real vector subspaces of a quaternionic vector space, 
\textit{Proc. Edinb. Math. Soc. (2)}, {\bf 56} (2013) 615--622.

\bibitem{Pan-twistor_(co-)cr_q} 
R.~Pantilie, On the twistor space of a (co-)CR quaternionic manifold, 
\textit{New York J. Math.}, {\bf 20} (2014) 959--971. 

\bibitem{Pan-qgfs} 
R.~Pantilie, On the embeddings of the Riemann sphere with nonnegative normal bundles, Preprint IMAR, Bucharest 2013, 
(available from \href{http://arxiv.org/abs/1307.1993}{\tt http://arxiv.org/abs/1307.1993}). 

\bibitem{PanWoo-sd}
R.~Pantilie, J.~C.~Wood, Twistorial harmonic morphisms with one-dimensional fibres on self-dual four-manifolds,
\textit{Q. J. Math}, {\bf 57} (2006) 105--132.

\bibitem{Qui-QJM98}
D.~Quillen, Quaternionic algebra and sheaves on the Riemann sphere,
\textit{Q. J. Math.}, {\bf 49} (1998) 163--198. 

\bibitem{Ro-LeB_nonrealiz}
H.~Rossi, LeBrun's nonrealizability theorem in higher dimensions,
\textit{Duke Math. J.}, {\bf 52} (1985) 457--474.

\bibitem{Swa-99} 
A.~Swann, Homogeneous twistor spaces and nilpotent orbits, 
\textit{Math. Ann.}, {\bf 313} (1999) 161--188.



\end{thebibliography}
\end{document}